\documentclass[11pt,reqno]{amsart}
\usepackage[english]{babel}
\usepackage{lmodern}
\usepackage[top=1in,bottom=1in,left=1in,right=1in,marginparwidth=1.7cm,marginparsep=0.2cm]{geometry}
\usepackage{amsmath,amssymb,amsthm,mathtools}
\usepackage[dvipsnames]{xcolor}
\usepackage{booktabs,array,tikz}
\usepackage{hyperref}
\hypersetup{colorlinks=true,linkcolor=blue,filecolor=magenta,urlcolor=red,citecolor=blue}
\hypersetup{pdftitle={Simple Cuts, Simple Landscape: Dimension-Free Rank Bounds for Burer--Monteiro Factorizations via Rank-Width},pdfsubject={A binary block-energy induction for signed interactions}}
\allowdisplaybreaks[1]
\theoremstyle{plain}
\newtheorem{theorem}{Theorem}[section]
\newtheorem{lemma}[theorem]{Lemma}

\newtheorem{definition}[theorem]{Definition}

\newcommand{\R}{\mathbb R}
\newcommand{\tr}{\operatorname{tr}}
\newcommand{\rank}{\operatorname{rank}}
\newcommand{\diag}{\operatorname{diag}}
\newcommand{\im}{\operatorname{range}}
\newcommand{\Span}{\operatorname{span}}
\newcommand{\rw}{\operatorname{rw}_{\R}}
\newcommand{\ip}[2]{\langle #1,#2\rangle}
\newcommand{\norm}[1]{\lVert #1\rVert}
\newcommand{\D}{\Delta}
\newcommand{\Qs}{\mathcal Q_S}

\title[Simple Cuts, Simple Landscape]{Simple Cuts, Simple Landscape:\\Dimension-Free Rank Bounds for\\Burer--Monteiro Factorizations via Rank-Width}
\author[Wu]{Hongjin Wu}
\date{}
\hypersetup{pdfauthor={Hongjin Wu}}
\begin{document}
\raggedbottom
\begin{abstract}
In this work, we study Burer--Monteiro factorizations of semidefinite
programs with unit diagonal constraints. For a real symmetric cost matrix
whose rank-width over $\R$ is at most $k$, we prove that every non-flat
second-order stationary point of the factorized problem is globally optimal whenever the
factorization rank is at least $2k$, independently of the matrix size.
\end{abstract}
\maketitle

\section{Introduction}\label{sec:introduction}

Low-rank matrix factorizations can reduce the number of variables in
semidefinite optimization, but introduce nonconvexity. A central question
is how large the factorization rank must be for second-order stationarity
to guarantee global optimality. In this work, we study semidefinite
programs with unit diagonal constraints and ask whether the required
rank can be controlled by the structure of the cost matrix, independently
of its size.

Let $A=A^\top\in\R^{n\times n}$ have zero diagonal. Consider the
semidefinite program
\begin{equation}\label{eq:sdp}
 \min_{Z=Z^\top\succeq0}
 \left\{-\frac12\tr(AZ): Z_{ii}=1,\ i=1,\ldots,n\right\}.
\end{equation}
The zero-diagonal assumption entails no loss of generality, since the
diagonal of $A$ contributes only a constant to the objective. Following
the Burer--Monteiro approach~\cite{BM03}, we write $Z=VV^\top$ with
$V\in\R^{n\times m}$, thereby restricting the rank of $Z$ to at most
$m$. Positive semidefiniteness is then automatic, and the diagonal
constraints require the rows of $V$ to have unit norm. Writing the $i$th
row as $v_i^\top$, we obtain
\begin{equation}\label{eq:energy}
 \min_{v_1,\ldots,v_n\in\mathbb S^{m-1}} E(V),
 \qquad E(V)=-\sum_{i<j}A_{ij}\ip{v_i}{v_j}
           =-\frac12\tr(V^\top AV).
\end{equation}
The parameter $m$ is the factorization rank. Choosing a small $m$
reduces the number of variables, whereas increasing $m$ enlarges the
space in which the vectors can move. The optimization landscape
therefore depends on both $A$ and $m$.

General landscape guarantees relate the required factorization rank to
the number of constraints. For the problem considered here, a general
result implies that, for almost every cost matrix $A$, every second-order
stationary point of \eqref{eq:energy} yields a global optimum of
\eqref{eq:sdp} whenever
\[
 \frac{m(m+1)}{2}>n
\]
\cite{BVB}. This sufficient condition depends on the number of diagonal
constraints, without distinguishing cost matrices according to their
interaction structure. Moreover, an almost-everywhere guarantee in the
ambient space does not by itself establish a guarantee for every matrix
in a prescribed structured family. These observations motivate the
following question.

\par\medskip
\noindent\textbf{Question.}
Can the structure of the cost matrix yield sharper bounds on the
factorization rank required for every second-order stationary point
to be globally optimal?
\par\medskip

In this work, we study this question through rank-width over $\R$. For a subset
$S\subseteq[n]$, the interaction between $S$ and its complement is
described by the cut matrix $A[S,[n]\setminus S]$. A rank-decomposition
is a tree of maximum degree three whose leaves are identified with
$[n]$. Deleting an edge partitions the leaves into two sets and hence
determines a cut matrix. The width of the decomposition is the largest
rank, over $\R$, of these cut matrices. The rank-width of $A$ over $\R$,
denoted by $\rw(A)$, is the minimum width over all such decompositions.
Thus, bounded rank-width over $\R$ expresses a form of structural simplicity:
the interactions across the cuts of a suitable decomposition have low
rank, even when the matrix is large or dense.

\par\medskip
\noindent\textbf{Main result (informal).}
If $\rw(A)\le k$ and $m\ge2k$, then every non-flat second-order
stationary point of \eqref{eq:energy} is a global minimizer of
\eqref{eq:energy}.
\par\medskip

The formulation also connects directly to the global synchronization problem on graphs. A positive
coefficient $A_{ij}$ favors alignment of $v_i$ and $v_j$, whereas a
negative coefficient favors opposite directions. For $m=2$, writing
$v_i=(\cos\theta_i,\sin\theta_i)$ gives
\[
 E(V)=-\sum_{i<j}A_{ij}\cos(\theta_i-\theta_j),
\]
the phase interaction energy arising in synchronization problems. The
factorization rank thus also describes the dimension of the interacting
unit vectors. From this perspective, our result relates the structure
of the interactions to the dimension sufficient to rule out nonglobal,
non-flat second-order stationary points.

\subsection{Main result}\label{subsec:contributions}

Our main result shows that a factorization rank of at least $2k$
suffices when the cost matrix has rank-width at most $k$.

\begin{theorem}\label{thm:main}
Let $k\ge1$ be an integer, and let $A\in\R^{n\times n}$ be symmetric with zero diagonal
and rank-width over $\R$ at most $k$. If $m\ge2k$, every non-flat
second-order stationary point of \eqref{eq:energy} is a global minimizer.
\end{theorem}

\subsection{Related work}\label{subsec:related-work}

Although semidefinite programs are convex, directly optimizing over a
dense $n\times n$ symmetric matrix involves $n(n+1)/2$ scalar variables
and can become computationally expensive as $n$ grows. To reduce this
cost, Burer and Monteiro~\cite{BM03} proposed the low-rank factorization
$Z=VV^\top$, with $V\in\R^{n\times m}$ and $m\ll n$. This reduces the
number of variables to $nm$ and enforces positive semidefiniteness
automatically. The resulting problem is nonconvex, however, raising the
possibility of nonglobal local minima. A central question is therefore
how large $m$ must be to guarantee that every second-order stationary
point of the factorized problem is globally optimal.

Boumal, Voroninski, and Bandeira~\cite{BVB} established such guarantees for
smooth equality-constrained semidefinite programs. In the unit-diagonal
setting, their generic-cost guarantee applies when
$m(m+1)/2>n$, while their uniform guarantee for all cost matrices applies
when $m>n/2$~\cite[Theorem~1.4 and Corollary~5.11]{BVB}.
Waldspurger and Waters~\cite{WW20} showed that, in general, this generic
rank threshold must grow on the order of the square root of the number
of constraints, even for problems with a rank-one global solution.
A bound independent of $n$ therefore calls for additional structure.

The bounds above depend on the matrix size $n$. For matrices of
rank-width at most $k$, Theorem~\ref{thm:main} gives the sufficient bound
$m\ge2k$, independent of $n$. We impose no genericity assumption on the
cost matrix $A$, but require the second-order stationary point $V^*$ to
be non-flat.

Our argument starts from the rank-width-one synchronization mechanism
in Wu~\cite{WuPositive}. For nonnegative weights and factorization
rank $m=2$, that work proves that every local minimum is synchronized
when the coupling matrix has connected support and rank-width one over
$\R$. The proof merges synchronized vertex sets along a
rank-decomposition tree. Since synchronized configurations are globally
optimal in that setting, proving synchronization establishes global
optimality.

Signed weights introduce a different difficulty: a global minimizer
need not be synchronized, and its configuration is not known in advance.
We therefore compare the energy at the stationary configuration with
the energy at an arbitrary feasible configuration.
We split this difference into the changes obtained by moving each side
of a cut separately, minus a cross term. We prove lower bounds for the
two single-set changes by induction on a rank-decomposition tree.
Together, these bounds are large enough to offset the cross term, so
the total energy cannot decrease.

\subsection{Organization}\label{subsec:organization}

Section~\ref{sec:setup} introduces the first- and second-order conditions
and non-flatness. We then define rank-width over $\R$ and show how a
cut factorization expresses the interaction between two vertex sets
in terms of their aggregate vectors.
Section~\ref{sec:merge} establishes the energy bound used in the proof.
The key step is an induction argument: bounds for the energy changes
on two child sets are combined to obtain the corresponding bound on
their union. We first give this argument for rank-width one, then
extend it to rank-width $k\ge2$.
Section~\ref{sec:finish} applies the bounds to the two sides of the
root cut. Together, they control the cross term in the total energy
change and show that no feasible configuration has lower energy than
$V^*$, completing the proof of Theorem~\ref{thm:main}.

\section{Preliminaries}\label{sec:setup}

\subsection{Notation}\label{subsec:notation}

\noindent\textbf{Sets and matrices.}
Write $[n]=\{1,\ldots,n\}$ and $S^c=[n]\setminus S$. For index sets $S,T$,
$A[S,T]$ denotes the corresponding submatrix, and $V_S$ consists of the
rows of $V$ indexed by $S$. The notation $L\sqcup R$ means a disjoint
union. All matrix ranks are taken over $\R$. For a symmetric matrix,
$C\succeq0$ and $C\succ0$ mean positive semidefiniteness and positive
definiteness, respectively.

\smallskip
\noindent\textbf{Vectors.}
We use $\ip{\cdot}{\cdot}$ and $\norm{\cdot}$ for the Euclidean inner
product and norm, and $\norm{\cdot}_F$ for the Frobenius norm.
Consider a change from a fixed original configuration $V^*$ to another
feasible configuration $V$. At each vertex $i$, the vector changes from
$v_i^*$ to $v_i$; we denote its change by $\D v_i=v_i-v_i^*$.

\subsection{Stationary points and non-flat directions}\label{subsec:stationarity}

We call a feasible configuration $V^*$ a \emph{second-order stationary
point} (SOSP) if
\[
 \operatorname{grad}E(V^*)=0,
 \qquad \operatorname{Hess}E(V^*)\succeq0,
\]
where $\operatorname{grad}$ and $\operatorname{Hess}$ are the Riemannian
gradient and Hessian on $(\mathbb S^{m-1})^n$, and positive semidefiniteness
is understood on the tangent space at $V^*$.
We also impose the following condition on motions that leave the
energy unchanged.

\begin{definition}[Non-flatness]\label{def:nonflat}
A configuration $V^*$ is \emph{non-flat} if the following holds for every
choice of fixed skew-symmetric matrices $\Omega_i\in\R^{m\times m}$.
If the energy remains constant along
\[
 v_i(t)=e^{t\Omega_i}v_i^*
\]
for all sufficiently small $t$, then every inner product
$\ip{v_i(t)}{v_j(t)}$ also remains constant for all sufficiently small $t$.
\end{definition}
Thus a motion of this form can preserve the energy only if it preserves
all pairwise inner products. Rotating all vectors together is allowed.
When $m=2$, write
$\Omega_i=\omega_i J$, where $J$ is the generator of planar rotations.
The paths become $\theta_i(t)=\theta_i^*+t\omega_i$. Their full Gram
matrix is constant precisely when all $\omega_i$ agree. Thus the
definition agrees with non-flatness along straight phase directions in
the planar problem. The condition concerns paths along which the energy
is exactly constant; the Hessian may still have additional zero directions.

\begin{lemma}[Stationarity conditions]\label{lem:energy}
Define
\begin{equation}\label{eq:M}
 \mu_i=\sum_jA_{ij}\ip{v_i^*}{v_j^*},
 \qquad M=\diag(\mu)-A.
\end{equation}
At an SOSP,
\begin{equation}\label{eq:foc-soc}
 MV^*=0,\qquad
 \tr(\Xi^\top M\Xi)\ge0
 \quad\text{whenever }\ip{\xi_i}{v_i^*}=0\text{ for every }i.
\end{equation}
\end{lemma}
\begin{proof}
First-order stationarity on each sphere gives
$\sum_jA_{ij}v_j^*=\mu_i v_i^*$, hence $MV^*=0$.
For a feasible curve, write $\xi_i=v_i'(0)$ and $a_i=v_i''(0)$.
Differentiating $\norm{v_i(t)}^2=1$ twice gives
$\ip{a_i}{v_i^*}=-\norm{\xi_i}^2$. Substituting this into the second
derivative of the energy at a stationary point gives
\[
 E''(0)=\sum_i\mu_i\norm{\xi_i}^2
        -\sum_{i,j}A_{ij}\ip{\xi_i}{\xi_j}
       =\tr(\Xi^\top M\Xi).
\]
Every tangent direction is the velocity of a feasible curve. The
second-order condition therefore gives \eqref{eq:foc-soc} for every
matrix $\Xi$ whose rows satisfy $\xi_i\perp v_i^*$.

\end{proof}

\begin{lemma}[Nonzero cuts and positive single-vertex coefficients]\label{lem:positive-leaves}
Suppose $n\ge2$ and $V^*$ is a non-flat SOSP with $m\ge2$.
Every nontrivial cut of $A$ is nonzero, and $\mu_i>0$ for every node $i$.
\end{lemma}
\begin{proof}
If $A[S,S^c]=0$ for a nontrivial cut, rotating all vectors in $S$
together preserves the energy: the internal inner products stay fixed,
and there is no interaction across the cut. Such a rotation can still
change an inner product between a node in $S$ and one in
$S^c$: for any two nonzero vectors in dimension at least two, a suitable
plane rotation of the first makes their inner product nonconstant.
This contradicts non-flatness.

Next, choose a nonzero tangent direction that moves only node $i$.
The second-order condition \eqref{eq:foc-soc} gives $\mu_i\ge0$.
If $\mu_i=0$, first-order stationarity gives
$\sum_jA_{ij}v_j^*=0$. Every rotation of node $i$ alone then preserves
the energy. A suitable rotation changes an inner product with another
node, again a contradiction.
\end{proof}

\subsection{Rank-width}\label{subsec:rank-width}

For $S\subseteq[n]=\{1,\ldots,n\}$, let
\[
 \rho_A(S)=\rank_{\R}A[S,S^c].
\]
A rank-decomposition is a tree of maximum degree three with one leaf
for each vertex in $[n]$. Removing an edge separates the leaves into
two sets $S$ and $S^c$. The width of that edge is $\rho_A(S)$.
The \emph{rank-width over $\R$}
$\rw(A)$ is the minimum, over such trees, of the maximum edge width.
We set $\rw(A)=0$ for $n\le1$. A degree-two internal vertex can be
removed by replacing its two incident edges with a single edge; this
does not change the cuts represented by the tree. We use the matrix
rank-width framework of Oum~\cite{Oum12},
based on the tree-and-cut construction of Oum and Seymour~\cite{OS}.
Throughout this paper, the field is $\R$. For an unweighted graph,
this need not agree with rank-width computed over $\mathbb F_2$.

To root the decomposition, first remove all degree-two internal vertices.
Then insert a new vertex on one edge and use it as the root.
Each internal vertex now has two children. For a vertex other than the
root, let $S$ be the set of leaves below it. The edge joining this vertex
to its parent separates $S$ from $S^c$. The induction proves a bound for
each such set $S$, starting with individual leaves and then joining the
sets below two children.

\subsection{Aggregate vectors across a cut}\label{subsec:aggregates}

Let $S\subseteq[n]$ and suppose the cut matrix $A[S,S^c]$ has rank
$r_S\ge1$. We will express the energy across this cut using $r_S$
weighted sums of vectors on each side. Choose a factorization
\begin{equation}\label{eq:cut}
 A[S,S^c]=P_SQ_S^\top,\qquad
 P_S\in\R^{|S|\times r_S},\quad
 Q_S\in\R^{|S^c|\times r_S}.
\end{equation}
Both factors have full column rank. For $a=1,\ldots,r_S$, the
$a$th column of $P_S$ defines an aggregate vector on $S$. Collect these
vectors as the columns of $X_S$:
\begin{equation}\label{eq:aggregates}
 x_S^a=\sum_{i\in S}(P_S)_{ia}v_i,\qquad
 X_S=V_S^\top P_S=[x_S^1\ \cdots\ x_S^{r_S}]
       \in\R^{m\times r_S}.
\end{equation}
When only $S$ moves, the vectors in $S^c$ remain at their original
values. Using the columns of $Q_S$ as weights, define their aggregates by
\[
 \begin{aligned}
 (w_S^a)^*&=\sum_{j\in S^c}(Q_S)_{ja}v_j^*,
       \qquad a=1,\ldots,r_S,\\
 W_S^*&=[(w_S^1)^*\ \cdots\ (w_S^{r_S})^*]
        =(V_{S^c}^*)^\top Q_S.
 \end{aligned}
\]
Substituting the factorization into the crossing energy gives
\begin{equation}\label{eq:aggregate-crossing}
 -\sum_{i\in S,j\in S^c}A_{ij}\ip{v_i}{v_j^*}
 =-\sum_{a=1}^{r_S}\ip{x_S^a}{(w_S^a)^*}
 =-\tr\bigl(X_S(W_S^*)^\top\bigr).
\end{equation}
Thus the crossing energy is a sum of $r_S$ inner products between
aggregate vectors. In particular, a rank-one cut requires one aggregate
vector on each side. The weights in $P_S$ and $Q_S$ may have either sign.

We keep these factors fixed when comparing configurations. The change
in the aggregates is therefore
\[
 \D X_S=X_S-X_S^*=(\D V_S)^\top P_S.
\]
When two sets $L$ and $R$ are joined to form $S$, we need to relate their
aggregate vectors. The next lemma expresses the aggregates of $S$ and
the interaction between $L$ and $R$ using the aggregates of the two children.

\begin{lemma}[Relating the aggregates of a set and its children]\label{lem:compatible}
Let $S$ be the leaf set below a nonroot vertex, with children whose leaf
sets are $L$ and $R$, so $S=L\sqcup R$. Choose full-column-rank cut
factorizations for $L,R,S$. There are matrices
\[
 \Gamma\in\R^{r_L\times r_R},\qquad
 H\in\R^{(r_L+r_R)\times r_S}
\]
such that, with $P_0=\diag(P_L,P_R)$,
\begin{equation}\label{eq:compatible}
 A[L,R]=P_L\Gamma P_R^\top,\qquad P_S=P_0H.
\end{equation}
The matrix $H$ has full column rank, and, writing $X_0=[X_L\ X_R]$,
\begin{equation}\label{eq:aggregate-merge}
 X_S=X_0H.
\end{equation}
\end{lemma}
\begin{proof}
The matrix $A[L,R]$ is part of the cut matrix for $L$. Its columns
therefore lie in $\im P_L$. Applying the same argument to the cut for
$R$ shows that its rows lie in the row space of $P_R^\top$. These two
inclusions give $A[L,R]=P_L\Gamma P_R^\top$ for some $\Gamma$.
Similarly, the rows of $P_S$ indexed by $L$ can be written as $P_L$
times a coefficient matrix, and the rows indexed by $R$ can be written
using $P_R$. Stacking these two coefficient matrices gives $H$.
The full column rank of $P_S=P_0H$ implies the full column rank of $H$.
Multiplication by $V_S^\top$ proves \eqref{eq:aggregate-merge}.
\end{proof}

\begin{lemma}[Alignment across a rank-one cut]\label{lem:scalar-alignment}
Assume $m\ge2$. At a non-flat SOSP, let $A[S,S^c]=pq^\top$ be a nonzero rank-one cut.
Put $u^*=\sum_{i\in S}p_i v_i^*$ and
$w^*=\sum_{j\notin S}q_jv_j^*$. Then
\begin{equation}\label{eq:scalar-alignment}
 u^*\ne0,\qquad w^*\ne0,\qquad w^*=\lambda_Su^*,\quad\lambda_S>0.
\end{equation}
\end{lemma}
\begin{proof}
If either aggregate were zero, rotating all vectors in $S$ together
would preserve the crossing energy $-\ip{u^*}{w^*}$ and all internal
energies. We could choose a rotation that changes an inner product
between a vector in $S$ and one in $S^c$, violating non-flatness.
Thus both aggregates are nonzero.

Next, rotate the vectors in $S$ by $e^{t\Omega}$ while fixing $S^c$.
The first-order condition gives
$\ip{\Omega u^*}{w^*}=0$. If $w^*$ had a nonzero component
$a\perp u^*$, choose
$\Omega=(a(u^*)^\top-u^*a^\top)/\norm{u^*}^2$ so that
$\Omega u^*=a$. The first-order condition would then give
$\norm{a}^2=0$, a contradiction. Thus $w^*=\lambda_Su^*$ for some scalar
$\lambda_S$. To determine its sign, rotate $S$ at unit angular speed in
a plane containing $u^*$. The second derivative of the crossing energy
is $\lambda_S\norm{u^*}^2$, which must be nonnegative. Since both
aggregates are nonzero, $\lambda_S>0$.
\end{proof}

\section{Splitting and lower bounding energy changes}\label{sec:merge}

Fix a non-flat second-order stationary point $V^*$ and an arbitrary
feasible configuration $V$. We will prove that $E(V)\ge E(V^*)$.

\subsection{The split of energy changes}\label{sec:identities}

We first compare moving two disjoint sets $L$ and $R$ together with
moving each set separately from the same original configuration $V^*$.
The difference is a cross term involving the changes on both sides.
Taking $L\cup R=[n]$ then gives a decomposition of the total energy
change from $V^*$ to $V$.

To better present this decomposition, we start with some notation.
The matrix $V\in\R^{n\times m}$ records the vectors at the $n$ vertices:
its $i$th row is $v_i^\top$, where $v_i\in\R^m$ and $\norm{v_i}=1$.
For $S\subseteq[n]$, let $V_S$ consist of the rows indexed by $S$.
The total energy of this configuration is
\[
 E(V)=-\sum_{1\le i<j\le n}A_{ij}\ip{v_i}{v_j}
     =-\frac12\tr(V^\top AV).
\]
Define the energy of the edges within $S$ by
\[
 E_S(V_S)=-\sum_{\substack{i,j\in S\\i<j}}A_{ij}\ip{v_i}{v_j},
\]
and the energy of the edges joining $S$ to $S^c$ by
\begin{equation}\label{eq:subset-energies}
 E_{S\leftrightarrow S^c}(V_S,V_{S^c})
 =-\sum_{i\in S,\,j\in S^c}A_{ij}\ip{v_i}{v_j}.
\end{equation}
The total energy is therefore
\begin{equation}\label{eq:energy-partition}
 E(V)=E_S(V_S)+E_{S^c}(V_{S^c})
      +E_{S\leftrightarrow S^c}(V_S,V_{S^c}).
\end{equation}
The following figure illustrates how the cut $(S,S^c)$ decomposes
$E(V)$ into these three terms.
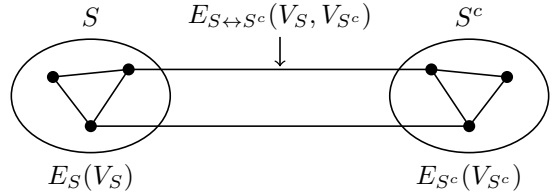
\begin{center}
\begin{tikzpicture}[
 x=1cm,y=1cm,line width=.6pt,
 every node/.style={font=\small},
 vertex/.style={circle,fill=black,inner sep=1.6pt}]
\draw (-2.5,0) ellipse (1.05 and .75);
\draw (2.5,0) ellipse (1.05 and .75);
\node at (-2.5,1.05) {$S$};
\node at (2.5,1.05) {$S^c$};

\coordinate (la) at (-3,.25);
\coordinate (lb) at (-2,.35);
\coordinate (lc) at (-2.5,-.4);
\coordinate (ra) at (2,.35);
\coordinate (rb) at (3,.25);
\coordinate (rc) at (2.5,-.4);
\draw (la)--(lb)--(lc)--cycle;
\draw (ra)--(rb)--(rc)--cycle;
\draw (lb)--(ra);
\draw (lc)--(rc);
\foreach \p in {la,lb,lc,ra,rb,rc}
 \node[vertex] at (\p) {};

\node at (-2.5,-1.08) {$E_S(V_S)$};
\node at (2.5,-1.08) {$E_{S^c}(V_{S^c})$};
\node at (0,1.05) {$E_{S\leftrightarrow S^c}(V_S,V_{S^c})$};
\draw[->] (0,.78)--(0,.39);
\end{tikzpicture}
\end{center}
Fix an original configuration $V^*$ and a new feasible configuration $V$,
and write $\D v_i=v_i-v_i^*$.
Moving only $S$ means replacing $v_i^*$ by $v_i$ for $i\in S$, while
keeping all vectors in $S^c$ at their original values.
We denote the total energy of this configuration by
$E(V_S,V_{S^c}^*)$ and define
\begin{equation}\label{eq:delta-def}
\begin{aligned}
 \D E_S
 &=E(V_S,V_{S^c}^*)-E(V_S^*,V_{S^c}^*)\\
 &=\bigl[E_S(V_S)-E_S(V_S^*)\bigr]\\
 &\quad+\bigl[E_{S\leftrightarrow S^c}(V_S,V_{S^c}^*)
   -E_{S\leftrightarrow S^c}(V_S^*,V_{S^c}^*)\bigr].
\end{aligned}
\end{equation}
The energy of the edges entirely within $S^c$ remains unchanged and
cancels in the difference. Thus, $\D E_S$ consists of the internal
energy change within $S$ and the energy change across the cut. In particular,
$\D E_{[n]}=E(V)-E(V^*)$.

\begin{lemma}\label{lem:finite-energy}
For disjoint vertex sets $L,R\subseteq[n]$, one has
\begin{equation}\label{eq:intro-merge}
 \D E_{L\cup R}=\D E_L+\D E_R
 -\sum_{i\in L,\,j\in R}A_{ij}\ip{\D v_i}{\D v_j}.
\end{equation}
\end{lemma}
\begin{proof}
Starting from $V^*$, first move the vectors in $L$ to their new values,
and then move those in $R$. The first move changes the energy by
$\D E_L$, so the second changes it by $\D E_{L\cup R}-\D E_L$.

Compare this second change with $\D E_R$, which is obtained by moving
$R$ directly from $V^*$. The only difference is that the vectors in $L$
are now $v_i$ rather than $v_i^*$. Edges within $R$ and edges joining
$R$ to $[n]\setminus(L\cup R)$ therefore have the same energy changes
in both cases. For an edge joining $i\in L$ to $j\in R$, the change is
$-A_{ij}\ip{v_i}{\D v_j}$ after $L$ has moved, compared with
$-A_{ij}\ip{v_i^*}{\D v_j}$ when starting from $V^*$. Their difference is
\[
 -A_{ij}\ip{v_i}{\D v_j}
 +A_{ij}\ip{v_i^*}{\D v_j}
 =-A_{ij}\ip{\D v_i}{\D v_j}.
\]
Summing over the edges between $L$ and $R$ gives
\[
 \D E_{L\cup R}-\D E_L-\D E_R
 =-\sum_{i\in L,\,j\in R}A_{ij}\ip{\D v_i}{\D v_j},
\]
which rearranges to \eqref{eq:intro-merge}.
\end{proof}

\subsection{Lower bounding \texorpdfstring{$\Delta E_S$}{the single-set energy change}}\label{subsec:induction-lemma}

\begin{lemma}[Lower bound when only a subtree moves]\label{lem:block-bound}
Let $V^*$ be a non-flat second-order stationary point under the
hypotheses of Theorem~\ref{thm:main}. We can choose a rank-decomposition
of width at most $k$ and a root with the following property. Let $S$ be
the set of leaves below any nonroot vertex, and choose any
full-column-rank factorization \eqref{eq:cut}. Then there is a fixed
symmetric positive definite matrix $T_S\in\R^{r_S\times r_S}$ such that
\begin{equation}\label{eq:block-bound}
 \D E_S\ge\frac12\tr(\D X_ST_S\D X_S^\top)
 \qquad\text{for every feasible new configuration.}
\end{equation}
\end{lemma}

\subsubsection{Rank-width one}\label{subsubsec:rank-one-bound}

\begin{proof}[Proof of Lemma~\ref{lem:block-bound} for $k=1$]
If $n=1$, there are no nonroot subtrees, so there is nothing to prove.
Hence assume $n\ge2$.
Since $\rw(A)\le1$, there exists a rank-decomposition tree in which
every edge induces a cut matrix of rank at most one. Every such cut is
nonzero by Lemma~\ref{lem:positive-leaves}, and hence has rank one. Thus
$r_S=1$, $X_S$ consists of a single aggregate vector $u_S$, and
\eqref{eq:block-bound} reduces to
\[
 \D E_S\ge\frac{\lambda_S}{2}\norm{\D u_S}^2,
 \qquad T_S=[\lambda_S].
\]
Remove degree-two internal vertices as described in
Section~\ref{subsec:rank-width}. Insert the root on any edge of the tree.
Each nonroot vertex represents the set $S$ of leaves below it, and each
internal vertex has two children. For each such $S$, write
$A[S,S^c]=p_Sq_S^\top$ and set
\[
 u_S=\sum_{i\in S}(p_S)_iv_i,
 \qquad w_S^*=\sum_{j\in S^c}(q_S)_jv_j^*.
\]
Lemma~\ref{lem:scalar-alignment} gives a fixed positive coefficient
$\lambda_S$ satisfying $w_S^*=\lambda_Su_S^*$. We prove the scalar
bound above with these coefficients.

\smallskip\noindent\textit{Base case.}
Let $S=\{i\}$ be the vertex set corresponding to a leaf.
The original aggregate vector is $u_S^*=(p_S)_iv_i^*$, and the new
aggregate vector is $u_S=(p_S)_iv_i$. Since $v_i^*$ and $v_i$ are
unit vectors and the weight $(p_S)_i$ is fixed,
\[
 \norm{u_S^*}=\norm{u_S}=|(p_S)_i|.
\]
There are no edges within $S$, so only the energy across the cut changes.
Using $w_S^*=\lambda_Su_S^*$, we obtain
\[
 \D E_S=-\ip{\D u_S}{w_S^*}
       =-\lambda_S\ip{\D u_S}{u_S^*}
       =\frac{\lambda_S}{2}\norm{\D u_S}^2.
\]
The last equality follows by expanding
$\norm{\D u_S}^2=\norm{u_S-u_S^*}^2$ and using
$\norm{u_S}=\norm{u_S^*}$. Thus the desired bound holds with equality
for a single vertex.

\smallskip\noindent\textit{Induction hypothesis.}
Let $S=L\sqcup R$ and assume the scalar bounds hold for both
children with their alignment coefficients $\lambda_L,\lambda_R$.

\smallskip\noindent\textit{Inductive step.}
Lemma~\ref{lem:compatible} gives scalar coefficients $c,d,\gamma$
such that $u_S=cu_L+du_R$ and
$A[L,R]=\gamma p_Lp_R^\top$. Using Lemma~\ref{lem:finite-energy} and
the two child bounds, we obtain
\begin{align*}
 \D E_S
 &=\D E_L+\D E_R-\gamma\ip{\D u_L}{\D u_R}\\
 &\ge\frac{\lambda_L}{2}\norm{\D u_L}^2
     +\frac{\lambda_R}{2}\norm{\D u_R}^2
     -\gamma\ip{\D u_L}{\D u_R}\\
 &=\frac{\lambda_S}{2}\norm{\D u_S}^2+Q.
\end{align*}
The two child bounds give the desired bound for $S$ plus a remainder
$Q$. It is therefore enough to prove $Q\ge0$. Expanding
$\norm{\D u_S}^2=\norm{c\D u_L+d\D u_R}^2$ gives
\begin{equation}\label{eq:scalar-merge}
 Q=\frac a2\norm{\D u_L}^2+\frac b2\norm{\D u_R}^2
       -\kappa\ip{\D u_L}{\D u_R},
\end{equation}
where
\[
 a=\lambda_L-\lambda_Sc^2,\quad
 b=\lambda_R-\lambda_Sd^2,\quad
 \kappa=\gamma+\lambda_Scd.
\]
The cut factorizations in Lemma~\ref{lem:compatible} give
$w_L^*=\gamma u_R^*+cw_S^*$ and
$w_R^*=\gamma u_L^*+dw_S^*$. Applying
Lemma~\ref{lem:scalar-alignment} to the three cuts yields
\[
 \gamma u_R^*+cw_S^*=\lambda_Lu_L^*,\quad
 \gamma u_L^*+dw_S^*=\lambda_Ru_R^*,\quad
 w_S^*=\lambda_S(cu_L^*+du_R^*).
\]
Substitute the third relation into the first two and collect the
coefficients of $u_L^*$ and $u_R^*$. This gives
\begin{equation}\label{eq:scalar-two-relations}
 a u_L^*=\kappa u_R^*,\qquad b u_R^*=\kappa u_L^*.
\end{equation}
If $u_L^*,u_R^*$ are linearly independent, the first equation forces
$a=\kappa=0$, and the second then gives $b=0$. Thus
$a=b=\kappa=0$, so $Q=0$.

If they are parallel, write $u_L^*=\rho e$, $u_R^*=\sigma e$, where
$\norm{e}=1$ and $\rho\sigma\ne0$; $\rho,\sigma$ may have either sign.
Set $\eta=\kappa/(\rho\sigma)$.
Equation~\eqref{eq:scalar-two-relations} gives
$a=\eta\sigma^2$, $b=\eta\rho^2$, and therefore
\begin{equation}\label{eq:scalar-Q-square}
 Q=\frac\eta2\norm{\sigma\D u_L-\rho\D u_R}^2.
\end{equation}
It remains to show that the coefficient $\eta$ is nonnegative.
Choose a two-dimensional plane containing $e$, and let $\mathcal R_t$
rotate through angle $t$ in this plane while fixing its orthogonal
complement. Apply $\mathcal R_{\alpha t}$ to every vector in $L$ and
$\mathcal R_{\beta t}$ to every vector in $R$, keeping the exterior fixed.
The node vectors need not all lie in the selected plane; their aggregate
vectors do. Internal energies of the two children remain unchanged.
The varying energy is
\[
 -\gamma\ip{\mathcal R_{\alpha t}u_L^*}{\mathcal R_{\beta t}u_R^*}
 -c\ip{\mathcal R_{\alpha t}u_L^*}{w_S^*}
 -d\ip{\mathcal R_{\beta t}u_R^*}{w_S^*}.
\]
At $t=0$, the second-order condition and the alignment equations give
\[
 0\le\lambda_L\rho^2\alpha^2+\lambda_R\sigma^2\beta^2
       -2\gamma\rho\sigma\alpha\beta.
\]
Choose $\alpha=d/\rho$ and $\beta=-c/\sigma$ to cancel the terms
containing $\lambda_S$. Then
\begin{align*}
 0&\le\lambda_Ld^2+\lambda_Rc^2+2\gamma cd\\
  &=ad^2+bc^2+2\kappa cd
   =\eta(c\rho+d\sigma)^2.
\end{align*}
By Lemma~\ref{lem:scalar-alignment}, the parent aggregate
$u_S^*=(c\rho+d\sigma)e$ is nonzero. Hence $\eta\ge0$ and
\eqref{eq:scalar-Q-square} gives $Q\ge0$, proving the parent bound.
This completes the induction for $k=1$.

\end{proof}

\subsubsection{Rank-width \texorpdfstring{$k\ge2$}{k >= 2}}\label{subsubsec:matrix-bound}

We first establish two geometric facts used in the induction.
The merge argument below uses directions perpendicular to the original
vectors in the set being merged. The first lemma bounds the dimension spanned by the vectors on one
side of a cut. The second uses this bound to choose a root so that
these perpendicular directions are available at all but one nonroot
merge when the original configuration spans $\R^m$.
Write $U_S=\Span\{v_i^*:i\in S\}$ for the space spanned by the original
vectors in $S$. Its orthogonal complement $U_S^\perp$ consists of the
vectors perpendicular to every $v_i^*$ with $i\in S$.

\begin{lemma}[The vectors in a set span the same space as its external forces]\label{lem:confinement}
Let $V^*$ be a non-flat stationary point. Suppose
$\rho_A(S)\le k$, $m\ge k+2$, and $\dim U_{S^c}>k$.
For each $i\in S$, let $h_i$ be the force from the vectors outside $S$,
and let $W$ be the space spanned by these forces:
\[
 h_i=\sum_{j\notin S}A_{ij}v_j^*,\qquad
 W=\Span\{h_i:i\in S\}.
\]
Then
\begin{equation}\label{eq:confinement}
 U_S=W\subseteq U_{S^c},\qquad \dim U_S\le k.
\end{equation}
\end{lemma}
\begin{proof}
The matrix whose rows are $h_i^\top$ is $A[S,S^c]V_{S^c}^*$.
Its rank is at most $k$, so $\dim W\le k$. Each $h_i$ is a linear
combination of exterior vectors, hence $W\subseteq U_{S^c}$.

We first show that every vector in $S$ belongs to $W$. Suppose instead
that some $v_i^*$ has a nonzero projection $a$ onto $W^\perp$.
Since $\dim U_{S^c}>k\ge\dim W$, some exterior vector also has a nonzero
projection $b$ onto $W^\perp$. Moreover, $m\ge k+2$ gives
$\dim W^\perp\ge2$, leaving room to rotate in this space.
There is a skew-symmetric operator $\Omega$ on $W^\perp$ such that
$\ip{e^{t\Omega}a}{b}$ is nonconstant. If $a,b$ are independent, take
$\Omega=ba^\top-ab^\top$, for which
\[
 \ip{\Omega a}{b}=\norm{a}^2\norm{b}^2-\ip{a}{b}^2>0.
\]
If $b=ca$ with $c\ne0$, rotate a plane containing $a$ inside $W^\perp$;
the inner product becomes $c\norm{a}^2\cos t$ after a choice of speed.

Extend $\Omega$ by zero on $W$ and rotate all nodes in $S$ with
$e^{t\Omega}$, fixing its exterior. Internal energies are unchanged,
and every crossing contribution $-\ip{e^{t\Omega}v_i^*}{h_i}$ is
unchanged because $h_i\in W$. However, the inner product between the
selected interior and exterior vectors changes. This contradicts
non-flatness, proving $U_S\subseteq W$.
For the reverse inclusion, first-order stationarity gives
\[
 h_i=\mu_i v_i^*-\sum_{j\in S}A_{ij}v_j^*\in U_S.
\]
This puts every $h_i$ in $U_S$, so $W\subseteq U_S$. Combining the two
inclusions proves \eqref{eq:confinement}.
\end{proof}

\begin{lemma}[Choosing the root when the original vectors span $\R^m$]\label{lem:root}
Suppose $k\ge2$, $m\ge2k$, $\rw(A)\le k$, and $V^*$ is a non-flat
stationary point with $\rank V^*=m$. Then $m=2k$.
Take any decomposition of width at most $k$ and remove its degree-two
internal vertices. There is an internal vertex whose three branches
have leaf sets
$S_1,S_2,S_3$ satisfying
\begin{equation}\label{eq:complementary}
 \dim U_{S_i}=k,\qquad
 U_{S_i}\oplus U_{S_j}=\R^{2k}\quad(i\ne j).
\end{equation}
Place the root on the edge separating $S_3$ from $S_1\cup S_2$.
Every nonroot merge then takes place in a set $S$ with
$\dim U_S^\perp\ge2$, except for
$S_1,S_2\longrightarrow S_1\cup S_2$.
\end{lemma}
\begin{proof}
We first show that every tree-edge cut has exactly one side whose
original vectors span more than $k$ dimensions. Since $m\ge2k\ge k+2$,
Lemma~\ref{lem:confinement} rules out the possibility that both sides
span more than $k$ dimensions.

Suppose both sides span at most $k$ dimensions. If $m>2k$, their sum
cannot be all of $\R^m$, contrary to $\rank V^*=m$. If $m=2k$, both
spaces must have dimension $k$ and their intersection must be zero.
For $i\in S$, the force $h_i$ belongs to $U_{S^c}$ by definition and
to $U_S$ by first-order stationarity. Hence $h_i=0$. Rotating $S$
together would then leave its energy unchanged, although a suitable
rotation changes an inner product across the cut. This contradicts
non-flatness.

Orient each edge toward its side of dimension greater than $k$.
Following the arrows in a finite tree must eventually lead to a vertex
with no outgoing edge. This vertex cannot be a leaf, since a single
vector spans only one dimension. It is therefore an internal vertex
with three incident edges, all pointing toward it. Each of its three
branches lies on the smaller side of its edge, so the original vectors
in each branch span at most $k$ dimensions.
For each branch, Lemma~\ref{lem:confinement} gives
\[
 U_{S_i}\subseteq U_{S_j}+U_{S_\ell}.
\]
Thus the vectors in any one branch lie in the space spanned by the other
two. Any two branches therefore span the entire original configuration,
which is $\R^m$. Each branch has dimension at most $k$, so $m\le2k$.
Together with $m\ge2k$, this gives $m=2k$. Each branch must have
dimension $k$, and the spaces of any two branches have zero intersection.
This proves \eqref{eq:complementary}.

Insert the root on the edge between $S_3$ and $S_1\cup S_2$.
Any subtree inside one branch has $\dim U_S\le k$, so
$\dim U_S^\perp\ge m-k=k\ge2$. Among the nonroot subtrees, the only
one that contains more than one branch is $S_1\cup S_2$. This is the
only merge that must be treated separately.
\end{proof}

\begin{figure}[htbp]
\centering
\begin{tikzpicture}[every node/.style={font=\small},line width=.6pt]
\node (root) at (0,2.2) {$[n]$};
\node (s) at (-2,1.2) {$S_1\cup S_2$};
\node (s3) at (2,1.2) {$S_3$};
\node (s1) at (-3,0) {$S_1$};
\node (s2) at (-1,0) {$S_2$};
\draw (root)--(s) (root)--(s3) (s)--(s1) (s)--(s2);
\node[anchor=east] at (-3.25,1.2) {$Q=0$};
\draw (-3,-.3)--(-3.55,-1.15)--(-2.45,-1.15)--cycle;
\draw (-1,-.3)--(-1.55,-1.15)--(-.45,-1.15)--cycle;
\draw (2,.9)--(1.45,.05)--(2.55,.05)--cycle;
\node at (-2,-1.55) {merges within a branch};
\node at (2,-.4) {merges within a branch};
\end{tikzpicture}
\caption{The root is placed on the edge separating $S_3$ from
$S_1\cup S_2$ so that the induction proceeds within each of the three
branches first. Within each branch, the original vectors span at most
$k$ dimensions, so every merge has at least two perpendicular directions
available. The only nonroot merge outside these branches is
$S_1,S_2\to S_1\cup S_2$; it is treated separately below, where we prove
$Q=0$.}
\label{fig:root}
\end{figure}
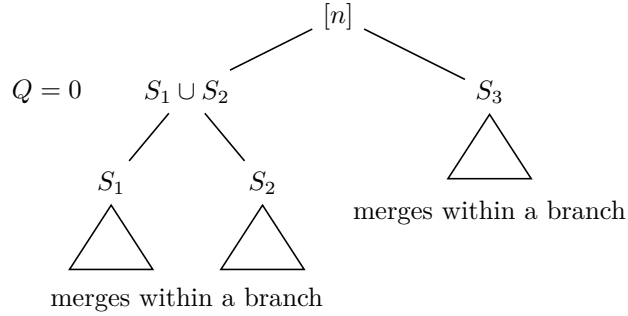

\begin{proof}[Proof of Lemma~\ref{lem:block-bound} for $k\ge2$]
If $n=1$, there are no nonroot subtrees, so there is nothing to prove.
Hence assume $n\ge2$.
Since $\rw(A)\le k$, there exists a rank-decomposition tree in which
every edge induces a cut matrix of rank at most $k$. Remove degree-two
internal vertices as described in Section~\ref{subsec:rank-width}.
If $\rank V^*<m$, choose any edge of this tree; if
$\rank V^*=m$, choose the edge specified by Lemma~\ref{lem:root}.
Insert a new vertex on the chosen edge and use it as the root.
We first prove the bound for each leaf, then use the bounds for two
children to prove the bound for their union.

\smallskip\noindent\textit{Base case.}
We first prove the bound for a single vertex $S=\{i\}$.
First-order stationarity gives
$\sum_jA_{ij}v_j^*=\mu_i v_i^*$, with $\mu_i>0$ by
Lemma~\ref{lem:positive-leaves}. Moving only this vertex therefore gives
\begin{align*}
 \D E_{\{i\}}
 &=-\left\langle v_i-v_i^*,\sum_jA_{ij}v_j^*\right\rangle\\
 &=-\mu_i\ip{v_i-v_i^*}{v_i^*}
   =\frac{\mu_i}{2}\norm{v_i-v_i^*}^2.
\end{align*}
The last equality uses $\norm{v_i}=\norm{v_i^*}=1$.
The cut factor is a nonzero scalar, $P_{\{i\}}=[p_i]$, so
$\D X_{\{i\}}=p_i\D v_i$. Taking
$T_{\{i\}}=[\mu_i/p_i^2]\succ0$ gives
\[
 \D E_{\{i\}}
 =\frac12\tr(\D X_{\{i\}}T_{\{i\}}\D X_{\{i\}}^\top),
\]
which establishes the base case with equality.

\smallskip\noindent\textit{Induction hypothesis.}
Let $S=L\sqcup R$, where $L$ and $R$ are the
leaf sets of its two children.
Assume that there exist fixed symmetric positive definite matrices
$T_L$ and $T_R$ such that
\[
 \D E_L\ge\frac12\tr(\D X_LT_L\D X_L^\top),
 \qquad
 \D E_R\ge\frac12\tr(\D X_RT_R\D X_R^\top)
\]
for every feasible new configuration.

\smallskip\noindent\textit{Inductive step.}
Lemma~\ref{lem:finite-energy} gives
\begin{align*}
 \D E_S
 &=\D E_L+\D E_R
   -\sum_{i\in L,\,j\in R}A_{ij}\ip{\D v_i}{\D v_j}\\
 &\ge\frac12\tr(\D X_LT_L\D X_L^\top)
     +\frac12\tr(\D X_RT_R\D X_R^\top)
     -\sum_{i\in L,\,j\in R}A_{ij}\ip{\D v_i}{\D v_j}.
\end{align*}
We want this right-hand side to be at least
$\tfrac12\tr(\D X_ST_S\D X_S^\top)$ for some fixed $T_S\succ0$.
Define the difference by
\begin{equation}\label{eq:merge-Q-expanded}
\begin{aligned}
 Q:={}&\frac12\tr(\D X_LT_L\D X_L^\top)
       +\frac12\tr(\D X_RT_R\D X_R^\top)\\
 &-\sum_{i\in L,\,j\in R}A_{ij}\ip{\D v_i}{\D v_j}
       -\frac12\tr(\D X_ST_S\D X_S^\top).
\end{aligned}
\end{equation}
Thus
\[
 \D E_S\ge\frac12\tr(\D X_ST_S\D X_S^\top)+Q.
\]
To obtain \eqref{eq:block-bound}, our goal boils down to finding a fixed
$T_S\succ0$ for which
$Q\ge0$ for every feasible new configuration.

Steps 1--3 establish the bound for a merge satisfying at least one of
the following conditions:
\begin{enumerate}
\item a nonzero vector is perpendicular to every original node vector;
\item $\dim U_S^\perp\ge2$.
\end{enumerate}
Both conditions give directions perpendicular to the original vectors
in $S$. We use them to test the second-order condition. Step 4 applies
the resulting bound throughout the tree, and Step 5 treats the one
remaining merge when the original vectors span $\R^m$.

\smallskip\noindent\textit{Step 1. We construct child coefficients satisfying \eqref{eq:invariant}.}
The induction hypothesis gives at least one positive definite coefficient
for each child. We will choose new coefficients that still satisfy the
child bounds. We also prove the matrix identities in
\eqref{eq:invariant}; these will let us apply the second-order
condition to the remainder $Q$.
For a child $I\in\{L,R\}$, denote the coefficient from the induction
hypothesis by $T_I^{(0)}$, and write $M_I=M[I,I]$.

We first express the exact energy change as a quadratic form in
$\D V_I$. Let $D$ equal $\D V_I$ on the rows in $I$ and zero on all
other rows, so $V^*+D$ is the configuration obtained by moving only $I$.
Expanding the energy gives
\[
 E(V^*+D)-E(V^*)
 =-\tr(D^\top AV^*)-\tfrac12\tr(D^\top AD).
\]
Write the $i$th row of $D$ as $d_i^\top$. Since both $v_i^*$ and
$v_i^*+d_i$ have unit length, $\ip{d_i}{v_i^*}=-\norm{d_i}^2/2$.
Using $AV^*=\diag(\mu)V^*$ in the linear term gives
\begin{equation}\label{eq:finite}
 \D E_I=\frac12\tr(\D V_I^\top M_I\D V_I).
\end{equation}

We next use small movements to turn the child bound into a matrix
inequality. Either condition (1) or (2) gives a unit vector $e$
perpendicular to every original vector in $I$. For any coefficients
$z\in\R^{|I|}$, move the vectors along
\[
 v_i(t)=\cos(tz_i)v_i^*+\sin(tz_i)e,\qquad i\in I,
\]
while fixing the exterior. Each vector stays on the unit sphere, and
its initial velocity is $z_i e$. Hence
$\D V_I(t)/t\longrightarrow ze^\top$. The child bound holds along this
curve. Substitute \eqref{eq:finite} into that bound, divide by $t^2$,
and let $t\to0$ to obtain
\[
 z^\top\bigl(M_I-P_IT_I^{(0)}P_I^\top\bigr)z\ge0.
\]
Since $z$ was arbitrary,
$M_I\succeq P_IT_I^{(0)}P_I^\top\succeq0$. If $M_Iz=0$, this inequality
and $T_I^{(0)}\succ0$ force $P_I^\top z=0$. Thus
$\ker M_I\subseteq\ker P_I^\top$: every column of $P_I$ is orthogonal
to $\ker M_I$ and therefore lies in the range of $M_I$.

The matrix $M_I$ may be singular, so we use its pseudoinverse
$M_I^\dagger$. For a positive semidefinite matrix, this replaces each
positive eigenvalue by its reciprocal and leaves each zero eigenvalue
at zero. Define the new coefficient $T_I$ and two auxiliary matrices by
\begin{equation}\label{eq:recover-coefficients}
 \begin{aligned}
 T_I&=(P_I^\top M_I^\dagger P_I)^{-1},\\
 C_I&=M_I-P_IT_IP_I^\top,\qquad
 Z_I=M_I^\dagger P_IT_I.
 \end{aligned}
\end{equation}
Here $C_I$ is the part of $M_I$ left after subtracting $P_IT_IP_I^\top$.
We will prove that $C_I\succeq0$, which ensures that the new coefficient
still gives a lower bound. The role of $Z_I$ is to produce perturbations
for which this remaining part contributes zero.

All three matrices depend only on $V^*$ and the fixed cut factors.
To check that $T_I$ exists and is positive definite, put
$D_I=M_I^{\dagger/2}P_I$, where $M_I^{\dagger/2}$ is the positive
semidefinite square root of $M_I^\dagger$. The columns of $P_I$ are
independent and belong to the range of $M_I$, so $D_I$ also has full
column rank. Therefore $T_I=(D_I^\top D_I)^{-1}\succ0$.
Moreover, $M_I^{1/2}D_I=P_I$, which gives
\[
 C_I=M_I^{1/2}
 \bigl[I-D_I(D_I^\top D_I)^{-1}D_I^\top\bigr]M_I^{1/2}\succeq0.
\]
The bracket is the orthogonal projection onto the space perpendicular
to the columns of $D_I$, so it is positive semidefinite.
Direct multiplication also gives $P_I^\top Z_I=I$ and
$M_IZ_I=P_IT_I$, hence $C_IZ_I=0$. In particular, for $z=Z_Ix$ we have
$P_I^\top z=x$ and $z^\top C_Iz=0$. These are the identities we will
use to test the second-order condition in Step 2.

We also need $C_IV_I^*=0$. The definition of $C_I$ shows that it
vanishes on $\ker M_I$ and satisfies $C_IM_I^\dagger P_I=0$.
First-order stationarity gives
$M_IV_I^*=P_I(W_I^*)^\top$, so
\[
 C_IV_I^*=C_IM_I^\dagger M_IV_I^*
         =C_IM_I^\dagger P_I(W_I^*)^\top=0.
\]
We have now proved all of the following identities:
\begin{equation}\label{eq:invariant}
 \begin{gathered}
 M_I=C_I+P_IT_IP_I^\top,\qquad T_I\succ0,\quad C_I\succeq0,\\
 C_IV_I^*=0,\qquad P_I^\top Z_I=I,\qquad C_IZ_I=0.
 \end{gathered}
\end{equation}
These identities follow from the child bounds. In particular, substituting
$M_I=C_I+P_IT_IP_I^\top$ into \eqref{eq:finite} shows that the new
coefficient gives the required bound for every feasible movement:
\begin{equation}\label{eq:certificate-energy}
 \D E_I
 =\frac12\tr(\D X_IT_I\D X_I^\top)
  +\frac12\tr(\D V_I^\top C_I\D V_I)
 \ge\frac12\tr(\D X_IT_I\D X_I^\top).
\end{equation}
Finally, substitute this decomposition of $M_I$ into
$M_IV_I^*=P_I(W_I^*)^\top$. Since $C_IV_I^*=0$ and $P_I$ has full
column rank, we can cancel $P_I$ and obtain
\begin{equation}\label{eq:alignment}
 W_I^*=X_I^*T_I.
\end{equation}

We can now combine the two child bounds. Use the new coefficients in
\eqref{eq:merge-Q-expanded}, and take $\Gamma,H$ from
Lemma~\ref{lem:compatible}. Collect the two sets of aggregate vectors
in $X_0=[X_L\ X_R]$, and collect their coefficients in
\begin{equation}\label{eq:J}
 J=\begin{pmatrix}T_L&-\Gamma\\-\Gamma^\top&T_R\end{pmatrix}.
\end{equation}
The crossing term is $\tr(\D X_L\Gamma\D X_R^\top)$, so the combined
child bound is
\begin{equation}\label{eq:child-lower}
 \D E_S\ge\frac12\tr(\D X_0J\D X_0^\top).
\end{equation}
Since $\D X_S=\D X_0H$, the remainder in
\eqref{eq:merge-Q-expanded} has the form
\begin{equation}\label{eq:merge-Q}
 \Qs=J-HT_SH^\top,\qquad
 Q=\frac12\tr(\D X_0\Qs\D X_0^\top).
\end{equation}
Here $\Qs$ is the coefficient matrix of the scalar remainder $Q$.
If we can choose $T_S\succ0$ so that $\Qs\succeq0$, then $Q\ge0$
for every change in the aggregate vectors, as required.

\smallskip\noindent\textit{Step 2. We show that $J\succeq0$ and $\ker J\subseteq\ker H^\top$.}
We first prove $J\succeq0$. Form the block diagonal matrices
\[
 C_0=\diag(C_L,C_R),\qquad Z_0=\diag(Z_L,Z_R),\qquad
 P_0=\diag(P_L,P_R).
\]
The diagonal blocks of $M[S,S]$ are $M_L$ and $M_R$, and its
off-diagonal block is $-A[L,R]=-P_L\Gamma P_R^\top$.
Using \eqref{eq:invariant} for the two children gives
\begin{equation}\label{eq:merge-matrix}
 M[S,S]=C_0+P_0JP_0^\top,\qquad
 P_0^\top Z_0=I,\qquad C_0Z_0=0,\qquad C_0V_S^*=0.
\end{equation}
To test $J$ in an arbitrary direction $x\in\R^{r_L+r_R}$, set $z=Z_0x$
and choose a unit vector $e\in U_S^\perp$. Move each vector in $S$
initially along $e$, with velocity
$\xi_i=z_ie$ for $i\in S$ and $\xi_j=0$ outside $S$.
These velocities are tangent to the spheres because $e\perp v_i^*$
for every $i\in S$. Using $C_0Z_0=0$ and $P_0^\top Z_0=I$, the
second-order condition becomes
\[
 0\le\tr(\Xi^\top M\Xi)
   =z^\top M[S,S]z=x^\top Jx.
\]
Thus $J\succeq0$.

We next show that $Jx=0$ implies $H^\top x=0$. This will allow us to
construct $T_S$ even when $J$ is singular.
If $x\in\ker J$, the perturbation just constructed has zero second
variation. The Hessian is positive semidefinite on the tangent space,
so its mixed term between this perturbation and any other tangent
direction must vanish. Indeed, nonnegativity on $\Xi+t\Upsilon$ for
both signs of $t$ rules out a nonzero linear term in $t$.
Choose $\Upsilon$ to move only an exterior node $j$, with velocity
$\zeta_j\perp v_j^*$. The mixed term gives
\begin{equation}\label{eq:mixed-zero}
 \left(\sum_{i\in S}A_{ji}z_i\right)\ip{e}{\zeta_j}=0.
\end{equation}
Under condition (1), choose $e$ perpendicular to all original vectors
and take $\zeta_j=e$. The inner product is then one, so the sum in
parentheses is zero. Under condition (2), we may choose any unit
$e\in U_S^\perp$. If the sum were nonzero, the equation would force
every such $e$ to be orthogonal to all $\zeta_j\perp v_j^*$, and hence
parallel to $v_j^*$. A space of dimension at least two cannot be
contained in this single line. The sum must again be zero.
Therefore $A[S^c,S]z=0$ in either case. Since $P_S=P_0H$,
\[
 0=A[S^c,S]z=Q_SH^\top P_0^\top Z_0x=Q_SH^\top x.
\]
Since $Q_S$ has full column rank, $Q_SH^\top x=0$ implies
$H^\top x=0$. As $x\in\ker J$ was arbitrary, we obtain
\begin{equation}\label{eq:kernel-condition}
 \ker J\subseteq\ker H^\top.
\end{equation}

\smallskip\noindent\textit{Step 3. We construct $T_S\succ0$ for which $Q\ge0$.}
By \eqref{eq:kernel-condition}, every column of $H$ is orthogonal to
$\ker J$ and therefore lies in the range of $J$. On that range,
$J^\dagger$ is positive definite. Since $H$ has full column rank,
$H^\top J^\dagger H\succ0$. Its inverse exists, so we choose
\begin{equation}\label{eq:ordinary-T}
 T_S=(H^\top J^\dagger H)^{-1}\succ0.
\end{equation}
We now check that this choice gives $Q\ge0$.
Put $D_S=J^{\dagger/2}H$. Then
$J^{1/2}D_S=H$ and $D_S^\top D_S=T_S^{-1}$, so
\begin{equation}\label{eq:projection}
 \Qs=J-HT_SH^\top
 =J^{1/2}\bigl[I-D_S(D_S^\top D_S)^{-1}D_S^\top\bigr]J^{1/2}
 \succeq0.
\end{equation}
The bracket projects onto the space perpendicular to the columns of
$D_S$ and is therefore positive semidefinite. Thus $\Qs\succeq0$, so
\eqref{eq:merge-Q} gives $Q\ge0$ for every feasible movement.
Substitution into the energy inequality proves the desired bound for $S$.

We also construct matrices $C_S,Z_S$ satisfying the same identities as
the child matrices in \eqref{eq:invariant}. These identities will be
used when combining bounds at the root. Set $G=J^\dagger HT_S$.
From the definition of $T_S$ and the fact that $H$ lies in the range of
$J$, we obtain
\begin{equation}\label{eq:G-identities}
 H^\top G=I,\qquad JG=HT_S,\qquad \Qs G=0.
\end{equation}
Define
\begin{equation}\label{eq:parent-CZ}
 C_S=C_0+P_0\Qs P_0^\top,\qquad Z_S=Z_0G.
\end{equation}
Then $C_S\succeq0$, $M[S,S]=C_S+P_ST_SP_S^\top$,
$P_S^\top Z_S=I$, and $C_SZ_S=0$.
It remains to check $C_SV_S^*=0$. First-order stationarity reads
$M[S,S]V_S^*=P_S(W_S^*)^\top$. Substitute \eqref{eq:merge-matrix},
use $C_0V_S^*=0$ and $P_S=P_0H$, and cancel the full-column-rank
matrix $P_0$. The result is
\begin{equation}\label{eq:J-stationarity}
 J(X_0^*)^\top=H(W_S^*)^\top,\qquad X_0^*=[X_L^*\ X_R^*].
\end{equation}
Multiplying by $G^\top$ gives
$(W_S^*)^\top=T_S(X_S^*)^\top$. Substituting this back yields
$\Qs(X_0^*)^\top=0$, and hence $C_SV_S^*=0$.
This verifies \eqref{eq:invariant} for $S$. The only induction
assumptions used were the two child energy bounds.

\smallskip\noindent\textit{Step 4. We prove the bounds along the rooted tree.}
If $\rank V^*<m$, a vector perpendicular to all original vectors
exists, so condition (1) applies at every nonroot merge. Starting at the leaves, Steps 1--3 therefore prove the bound for
every nonroot subtree.

Suppose instead that $\rank V^*=m$. By Lemma~\ref{lem:root}, condition
(2) applies to every merge inside the three branches $S_1,S_2,S_3$.
Each such subtree has vector span of dimension at most $k$, so its
orthogonal complement has dimension at least $m-k\ge2$. Starting at
the leaves, Steps 1--3 therefore establish the bounds throughout all
three branches. The only remaining nonroot set is $S_1\cup S_2$.

\smallskip\noindent\textit{Step 5. We prove the bound for $S_1\cup S_2$ with $Q=0$.}
Set $S=S_1\sqcup S_2$. We now construct $T_S\succ0$ with $Q=0$.
Each branch has vector span of dimension $k<m$, by
Lemma~\ref{lem:root}. Apply the coefficient construction
\eqref{eq:recover-coefficients} to the bound already established on each
branch. Since there is a vector perpendicular to all original vectors
in that branch, Step 1 gives matrices satisfying \eqref{eq:invariant}.
Write $X_i^*,W_i^*,T_i,C_i,Z_i$ for the data on $S_i$.
Lemma~\ref{lem:confinement} shows that the external forces on branch
$S_i$ span exactly $U_{S_i}$, a space of dimension $k$.
Their row matrix is $P_i(W_i^*)^\top$.
Since $P_i$ has full column rank, $W_i^*$ has rank $k$ and its column
space is $U_{S_i}$. The alignment $W_i^*=X_i^*T_i$ from
\eqref{eq:alignment}, together with $T_i\succ0$, gives
\begin{equation}\label{eq:branch-full-rank}
 r_i=k,\qquad \im X_i^*=U_{S_i}.
\end{equation}
By \eqref{eq:complementary}, the square matrix
$X_0^*=[X_1^*\ X_2^*]\in\R^{2k\times2k}$ is invertible.
Thus the columns of $X_1^*$ and $X_2^*$ together form a basis of
$\R^{2k}$. This is the fact that will allow us to prove $Q=0$.

The cut between $S$ and $S_3$ is the same cut viewed from opposite
sides. For $A[S_3,S]=P_3Q_3^\top$, choose the factors for $S$ as
$P_S=Q_3$ and $Q_S=P_3$, and set
\begin{equation}\label{eq:inverse-T}
 T_S=T_3^{-1}\succ0.
\end{equation}
Then
\[
 X_S^*=W_3^*=X_3^*T_3,
 \qquad W_S^*=X_3^*.
\]
Hence $W_S^*=X_S^*T_S$ for $T_S=T_3^{-1}\succ0$.
As in Steps 1--3, write $P_S=P_0H$ and form $J$ from $T_1,T_2$
and the interaction between $S_1$ and $S_2$.
As in \eqref{eq:J-stationarity}, first-order stationarity gives
\[
 J(X_0^*)^\top=H(W_S^*)^\top
             =HT_SH^\top(X_0^*)^\top.
\]
The invertibility of $X_0^*$ forces
\begin{equation}\label{eq:exceptional-zero}
 J-HT_SH^\top=0.
\end{equation}
Thus $Q=0$ for every movement. Define
\[
 C_S=\diag(C_1,C_2),\qquad
 Z_S=\diag(Z_1,Z_2)H(H^\top H)^{-1}.
\]
The inverse exists because $H$ has full column rank. The child
identities, \eqref{eq:exceptional-zero}, and direct multiplication
verify every condition in \eqref{eq:invariant}.
The bound follows from \eqref{eq:child-lower} and
\eqref{eq:merge-Q}, completing the induction for $k\ge2$.

Finally, any other full-column-rank factorization of the same cut
has the form $P_SB,Q_SB^{-\top}$ for an invertible matrix $B$.
Under this change, $X_S$ becomes $X_SB$. Replacing $T_S$ by
$B^{-1}T_SB^{-\top}$ preserves positive definiteness and leaves the
quadratic bound unchanged. The auxiliary relations are unchanged after replacing
$Z_S$ by $Z_SB^{-\top}$ and leaving $C_S$ fixed. Thus the claim holds
for every full-column-rank cut factorization.
\end{proof}

\section{Main proof of Theorem~\ref{thm:main}}\label{sec:proof}\label{sec:finish}

Let $V^*$ satisfy the assumptions of Theorem~\ref{thm:main}, and fix
an arbitrary feasible configuration $V$. We prove that
$E(V)-E(V^*)\ge0$. If $n=1$, the energy is constant, so assume $n\ge2$.
By Lemma~\ref{lem:positive-leaves}, every nontrivial cut is nonzero.

Choose the rooted rank-decomposition used in Lemma~\ref{lem:block-bound}.
Let $L$ and $R$ be the sets of vertices corresponding to the leaves
below the two children of the root. Then $[n]=L\sqcup R$.
Moving both sets therefore changes the entire configuration from $V^*$
to $V$. Lemma~\ref{lem:finite-energy} gives
\[
 E(V)-E(V^*)=\D E_L+\D E_R
       -\sum_{i\in L,\,j\in R}A_{ij}\ip{\D v_i}{\D v_j}.
\]
Lemma~\ref{lem:block-bound} gives lower bounds for $\D E_L$ and
$\D E_R$. It remains to show that their sum is at least the cross
term in this identity. We first do this for rank-width one, where
the two coefficients are positive scalars, and then for $k\ge2$.

\begin{proof}[Proof for $k=1$]
\textit{Step 1. We write the two bounds using the same cut factorization.}
Since $A[L,R]$ has rank one, write $A[L,R]=pq^\top$, and define
\[
 u=\sum_{i\in L}p_i v_i,\qquad
 w=\sum_{j\in R}q_j v_j.
\]
Write $u^*,w^*$ for these sums evaluated at $V^*$, and set
$\D u=u-u^*$ and $\D w=w-w^*$.
By Lemma~\ref{lem:scalar-alignment}, both original aggregates are
nonzero and $w^*=\lambda u^*$ for some $\lambda>0$.
Viewed from $R$, the same relation is $u^*=\lambda^{-1}w^*$.
The rank-one case of Lemma~\ref{lem:block-bound} therefore gives
\[
 \D E_L\ge\frac\lambda2\norm{\D u}^2,
 \qquad
 \D E_R\ge\frac1{2\lambda}\norm{\D w}^2.
\]

\smallskip\noindent\textit{Step 2. We substitute the bounds and complete the square.}
The factorization $A_{ij}=p_iq_j$ gives
\[
 \sum_{i\in L,\,j\in R}A_{ij}\ip{\D v_i}{\D v_j}
 =\left\langle\sum_{i\in L}p_i\D v_i,
                    \sum_{j\in R}q_j\D v_j\right\rangle
 =\ip{\D u}{\D w}.
\]
Substituting this expression and the two bounds into the total energy
change yields
\begin{align*}
 E(V)-E(V^*)
 &\ge\frac\lambda2\norm{\D u}^2
    +\frac1{2\lambda}\norm{\D w}^2-\ip{\D u}{\D w}\\
 &=\frac12\norm{\sqrt\lambda\,\D u-\lambda^{-1/2}\D w}^2\ge0.
\end{align*}
Since $V$ was arbitrary, $V^*$ is globally optimal.
\end{proof}

\begin{proof}[Proof for $k\ge2$]
We distinguish whether the original vectors span $\R^m$.
If they do not, a common perpendicular direction lets us test the
second-order condition at the root. If they do, the coefficients
constructed for the two root children are inverses, so we can complete
a square as in the rank-one case.

\smallskip\noindent\textbf{Case 1: The original vectors do not span $\R^m$.}
Thus $\rank V^*<m$. Fix a unit vector $e$ perpendicular to every
$v_i^*$, and insert the root on any decomposition edge.

\smallskip\noindent\textit{Step 1. We express the sum of the bounds minus the cross term as a quadratic form.}
Lemma~\ref{lem:block-bound} gives the bounds for the two root children
$L,R$. Apply the coefficient construction
\eqref{eq:recover-coefficients}, using $e$, to each of these bounds.
For $I=L,R$, we obtain $T_I\succ0$, $C_I\succeq0$, and $Z_I$ with
\[
 M[I,I]=C_I+P_IT_IP_I^\top,\qquad
 P_I^\top Z_I=I,\qquad C_IZ_I=0,
\]
and
\[
 \D E_I\ge\frac12\tr(\D X_IT_I\D X_I^\top).
\]
Here $M=\diag(\mu)-A$ is the matrix in
Lemma~\ref{lem:energy}. The matrices $C_I,Z_I$ will allow us to test
these coefficients using the second-order condition.

Choose $\Gamma$ so that $A[L,R]=P_L\Gamma P_R^\top$, as in
Lemma~\ref{lem:compatible}. Then the cross term is
$\tr(\D X_L\Gamma\D X_R^\top)$. Consequently, the two bounds and the
energy split give
\begin{align*}
 E(V)-E(V^*)
 &\ge\frac12\tr(\D X_LT_L\D X_L^\top)
      +\frac12\tr(\D X_RT_R\D X_R^\top)
      -\tr(\D X_L\Gamma\D X_R^\top)\\
 &=\frac12\tr\!\left(
 [\D X_L\ \D X_R]J_{\mathrm{root}}
 [\D X_L\ \D X_R]^\top\right),
\end{align*}
where
\[
 J_{\mathrm{root}}=
 \begin{pmatrix}T_L&-\Gamma\\-\Gamma^\top&T_R\end{pmatrix}.
\]
Thus it suffices to prove $J_{\mathrm{root}}\succeq0$.

\smallskip\noindent\textit{Step 2. We prove $J_{\mathrm{root}}\succeq0$ using the second-order condition.}
Group the child matrices into block diagonal matrices:
\[
 P_0=\diag(P_L,P_R),\qquad
 C_0=\diag(C_L,C_R),\qquad Z_0=\diag(Z_L,Z_R).
\]
Order the vertices with $L$ first and $R$ second. The identities in
Step 1 and $M[L,R]=-A[L,R]$ give
\[
 M=C_0+P_0J_{\mathrm{root}}P_0^\top,\qquad
 P_0^\top Z_0=I,\qquad C_0Z_0=0.
\]
Let $x\in\R^{r_L+r_R}$ be arbitrary and put $z=Z_0x$.
These identities imply $P_0^\top z=x$ and $C_0z=0$, so
\[
 z^\top Mz=x^\top J_{\mathrm{root}}x.
\]
To test the left-hand side, set $\xi_i=z_i e$ at every vertex.
Since $e\perp v_i^*$, each $\xi_i$ is tangent to the unit sphere at
$v_i^*$. Applying the second-order condition to the matrix $\Xi$
whose rows are $\xi_i^\top$, and using $\norm e=1$, gives
\[
 0\le\tr(\Xi^\top M\Xi)=z^\top Mz
   =x^\top J_{\mathrm{root}}x.
\]
As $x$ was arbitrary, $J_{\mathrm{root}}\succeq0$. Returning to the
inequality in Step 1, we conclude that
\begin{equation}\label{eq:root-deficient}
 E(V)-E(V^*)\ge
 \frac12\tr\!\left(
 [\D X_L\ \D X_R]J_{\mathrm{root}}
 [\D X_L\ \D X_R]^\top\right)\ge0.
\end{equation}
The last inequality holds because the trace is the sum of the
nonnegative quadratic forms obtained from the rows of
$[\D X_L\ \D X_R]$.

\smallskip\noindent\textbf{Case 2: The original vectors span $\R^m$.}
Thus $\rank V^*=m$. Lemma~\ref{lem:root} gives $m=2k$ and the
three sets $S_1,S_2,S_3$ in Figure~\ref{fig:root}. The root is placed
so that its two children correspond to $S=S_1\cup S_2$ and $S_3$.

\smallskip\noindent\textit{Step 1. We identify the coefficients for the two root children.}
Step 4 of the proof of Lemma~\ref{lem:block-bound} establishes the
bound on each of the three branches, including $S_3$. Step 5 constructs
the bound for $S$, with coefficient $T_S=T_3^{-1}$ as in
\eqref{eq:inverse-T}. Therefore,
\[
 \D E_S\ge\frac12\tr(\D X_ST_3^{-1}\D X_S^\top),
 \qquad
 \D E_{S_3}\ge\frac12\tr(\D X_3T_3\D X_3^\top).
\]
To see how the cross term is expressed in these same aggregates,
recall the factors chosen in Step 5:
$A[S_3,S]=P_3Q_3^\top$ and $A[S,S_3]=Q_3P_3^\top$.
Hence $\D X_S=(\D V_S)^\top Q_3$ and
$\D X_3=(\D V_{S_3})^\top P_3$, so
\[
 \sum_{i\in S,\,j\in S_3}A_{ij}\ip{\D v_i}{\D v_j}
 =\tr(\D X_S\D X_3^\top).
\]

\smallskip\noindent\textit{Step 2. We substitute the bounds and complete the square.}
Since $[n]=S\sqcup S_3$, Lemma~\ref{lem:finite-energy} gives
\[
 E(V)-E(V^*)=\D E_S+\D E_{S_3}
                   -\tr(\D X_S\D X_3^\top).
\]
Substitution of the two bounds yields
\begin{align}\label{eq:root-square}
 E(V)-E(V^*)
 &\ge\frac12\tr(\D X_ST_3^{-1}\D X_S^\top)
      +\frac12\tr(\D X_3T_3\D X_3^\top)
      -\tr(\D X_S\D X_3^\top)\notag\\
 &=\frac12\norm{\D X_ST_3^{-1/2}-\D X_3T_3^{1/2}}_F^2\ge0.
\end{align}
Here $T_3^{1/2}$ is the symmetric positive definite square root of
$T_3$. Expanding the squared norm gives exactly the two quadratic
terms and the cross term above.

Thus $E(V)\ge E(V^*)$ in both cases. Since $V$ was an arbitrary feasible
configuration, this proves Theorem~\ref{thm:main}.
\end{proof}

\section*{Declaration of AI assistance and formal verification}
The research idea and overall approach were developed by the author.
OpenAI's GPT-6 Astra was used to assist in developing and drafting the
proof of Lemma~\ref{lem:block-bound} for $k\ge2$, and in revising the
exposition. The author takes full responsibility for the content of the
manuscript.

Theorem~\ref{thm:main} and the energy lower bounds in
Lemma~\ref{lem:block-bound} have been formally verified in Lean~4 for
all parameter regimes stated therein.
The Lean source files and verification instructions are provided as
ancillary material.

\end{document}